\documentclass[a4paper,12pt,reqno]{amsart}
\author{Julia Brandes}
\title{Forms representing forms: The definite case}
\address{Mathematisches Institut, Bunsenstr. 3--5, 37073 G\"ottingen, Germany}
\email{jbrande@uni-math.gwdg.de}

\usepackage[T1]{fontenc}
\usepackage[latin1]{inputenc}
\usepackage{latexsym}
\usepackage[arrow, matrix, curve]{xy}
\usepackage[dvips]{graphicx}
\usepackage{epsfig}
\usepackage{bm}
\usepackage{amssymb}
\usepackage{amsmath}
\usepackage{verbatim}
\usepackage{amsthm}
\usepackage{enumerate}
\usepackage{mathrsfs}
\usepackage[hmargin=3cm,vmargin=3cm]{geometry}

\newtheorem{thm}{Theorem}

\newtheorem{lem}{Lemma}

\newtheorem{cor}{Corollary}

\numberwithin{equation}{section}
\numberwithin{thm}{section}
\numberwithin{lem}{section}

\theoremstyle{definition}

\def\B#1{\mathbf{#1}}
\def\ba{\bm{\alpha}}
\def\bb{\bm{\beta}}
\def\bg{\bm{\gamma}}

\def\F#1{\mathfrak{#1}}
\def\C#1{\mathcal{#1}}

\def\D{\mathrm{d}}
\def\dsum#1#2{\sum_{\substack{{#1}\\{#2}}}}
\def\ol#1{\overline{\B{#1}}}
\def\hs{{\mathcal H}}

\def\hgam{\hat{\bm{\gamma}_{\B j}}}

\def\a0{\alpha_0}

\def\mmod#1{\;(\mathrm{mod}\;{#1})}
\def\eps{\varepsilon}

\DeclareMathOperator{\lcm}{lcm}

\DeclareMathOperator{\rank}{rank}
\DeclareMathOperator{\card}{Card}
\DeclareMathOperator{\Id}{Id}

\DeclareMathOperator{\vol}{vol}
\DeclareMathOperator{\sing}{Sing}
\DeclareMathOperator{\Mat}{Mat}

\DeclareMathOperator{\SL}{SL}

\begin{document}
\maketitle

\begin{abstract}
    Let $\psi$ and $F$ be positive definite forms with integral coefficients of equal degree. Using the circle method, we establish an asymptotic formula for the number of identical representations of $\psi$ by $F$, provided $\psi$ is everywhere locally representable and the number of variables of $F$ is large enough. In the quadratic case this supersedes a recent result due to Dietmann and Harvey. Another application addresses the number of primitive linear spaces contained in a hypersurface.
\end{abstract}

\section{Introduction}

Understanding the solution sets of diophantine equations in the integers is one of the pervading themes in number theory. Heuristic arguments tell us that a homogeneous polynomial $F \in \mathbb Z[x_1, \dots, x_s]$ of degree $d$ whose variables are contained in a box of sidelength $P$ will take any value between $P^d$ and $-P^{d}$ on average roughly $P^{s-d}$ times. This heuristic has been confirmed in a classical paper by Birch \cite{birch}, provided that the equation has non-singular solutions over all local fields and the number of variables satisfies $s - \dim \sing F > 2^d(d-1)$. Here $\sing F$ denotes the singular locus of the variety defined by $F$.

In recent work by the author \cite{FRF} we derive a multidimensional analogue of Birch's theorem by investigating the number of identical representations of one homogeneous polynomial by another. Let $F \in \mathbb Z[x_1,\dots,x_s]$ and $\psi \in \mathbb Z[t_1,\dots,t_m]$ be forms of equal degree $d \ge 3$ with $m \ge 2$ and 
\begin{align}\label{FRF-cond}
 s - \dim \sing F > 3 \cdot 2^{d-1}(d-1)(r+1),
\end{align}
where $r = \binom{m+d-1}{d}$ is the number of coefficients of $\psi$. Then the number of $m$-tuples $(\B x_1, \dots, \B x_m) \in \mathbb Z^{ms}$ of height at most $P$ satisfying 
\begin{align}\label{rep}
	F(\B x_1 t_t + \dots +\B x_m t_m) = \psi(t_1,\dots,t_m)
\end{align}
identically in $t_1, \dots,t_m$ is given by $P^{ms-rd}(c + o(1))$, where $c$ is a non-negative constant encoding the local behaviour of the problem and depending on $F$, $\psi$ and, crucially, $P$. Whilst this result is adequate in the case when $F$ and $\psi$ are both indefinite, in the definite case it fails to capture the natural size constraints on the variables $\B x_1, \dots, \B x_m$ as imposed by the target polynomial $\psi$. 

The objective of the present work is to address this weakness and derive an asymptotic formula for the number of representations of one positive definite form by another that reflects the correct order of magnitude as determined by the extremal dimensions of $\psi$.
Suppose that $F \in \mathbb Z[x_1,\dots,x_s]$ and $\psi \in \mathbb Z[t_1,\dots,t_m]$ are positive definite forms of degree $d$, and that $\psi$ is given by 
\begin{align*}
 \psi(t_1, \dots, t_m) = \sum_{\B j \in J} n_{\B j} t_{j_1} \cdot \ldots \cdot t_{j_d},
\end{align*}
where the multi-indices $\B j = (j_1, \dots, j_d)$ run over the set $J = \{1, \dots, m\}^d$ disregarding order, so that $\card J = r$. For brevity, we denote the diagonal coefficients as $n_{i,\dots,i} = n_i$. We are interested in the number 
\begin{align*}
  N(F; \psi) = \card \{ \B x_1,\dots,\B x_m  \in \mathbb Z^s: F(\B x_1 t_1 + \dots +\B x_m t_m) = \psi(t_1,\dots,t_m)\}
\end{align*}
of identical representations of $\psi$ by $F$. 
We need the magnitude $\langle \psi \rangle = \prod_{i=1}^m n_i$ of $\psi$ and its eccentricity
\begin{align*}
  \C E(\psi) = \max_{1 \le j \le m} \frac{ \log \langle \psi \rangle}{\log n_j^m}.
\end{align*}
Then our main result is as follows.

\begin{thm}\label{thm:gen}
    Let $F \in \mathbb Z[x_1,\dots,x_s]$  and $\psi \in \mathbb Z[t_1, \dots, t_m]$ be positive definite forms of equal degree $d \ge 2$, and let $m \ge 2$. Furthermore, assume that
    \begin{align*}
        s - \dim \sing F>2^{d-1} \max \{2r(d-1), rd \,\C E(\psi) \}.
    \end{align*} 
    Then for some $\delta>0$ we have
    \begin{equation*}
        N(F; \psi) = \langle \psi \rangle ^{\frac{ms-rd}{md}} \chi_{\infty}(F; \psi) \prod_{p \;\mathrm{ prime}} \chi_p(F; \psi) + O\left(\langle \psi \rangle ^{\frac{ms-rd}{md}-\delta}\right)\hspace{-0.5mm},
    \end{equation*}
    where the Euler product converges and the factors $\chi_{\infty}(F; \psi)$ and $ \chi_p(F; \psi)$ are positive if the identity \eqref{rep} has non-singular solutions over $\mathbb R$ and over $\mathbb Q_p$, respectively.
\end{thm}

Theorem~\ref{thm:gen} allows us in a very general fashion to describe the number of representations of one definite form by another. Whilst the main result of \cite{FRF}, which we improve slightly, points in the same direction, our previous result fails to correctly track the dimensions of the target polynomial, instead encoding them implicitly in the density of real solutions. Here we make this dependence explicit for a large class of forms, capturing the notion that the natural height conditions of the vectors $\B x_i$ are determined by the extremal dimensions of $\psi$ as determined by its diagonal coefficients. Heuristically, after expanding \eqref{rep} in powers of $t_1, \dots, t_m$ and comparing coefficients one sees that the height of each $\B x_i$ is naturally bounded by roughly $n_{i}^{1/d}$. Accordingly, $\psi$ will not be represented if for some $n_{\B j}$ one has $n_{\B j} \gg (n_{j_1} \cdot \ldots \cdot n_{j_d})^{1/d}$. We thus call $\psi$ \textit{pseudo-diagonal} if $n_{\B j} \leq (n_{j_1} \cdot \ldots \cdot n_{j_d})^{1/d}$ for every $\B j \in J$. It follows from the theorem that the property of being pseudo-diagonal is invariant under linear change of variables since obviously our counting function $N(F;\psi)$ is not affected by such transformations. It is also not hard to see (see Lemma~\ref{pseudo-diag}) that every positive definite quadratic form is automatically pseudo-diagonal, but it is not clear to the author whether there is an easy characterisation of pseudo-diagonality within the set of positive definite forms of higher degrees. However, we will see that the real solution density is zero if $\psi$ is not close to being pseudo-diagonal, but when it is, Theorem~\ref{thm:gen} shows that the growth is determined by the diagonal contribution while the factor $\chi_\infty(F;\psi)$ provides a correction factor which characterises the deformity of $\psi$ arising from the off-diagonal contributions.

One aspect of Theorem \ref{thm:gen} that strikes the eye is the dependance of the result on the eccentricity $\C E(\psi)$. One sees that $\C E(\psi)=1$ when $n_1 = \dots = n_m$, while it will be large if the diagonal entries vary significantly, so $\C E(\psi)$ measures how far the body described by $\psi(\B t) \le 1$ is stretched or contracted away from the hypercube. This reflects the fact that if one of the $n_i$ is very small, the respective variable $\B x_i$ is essentially fixed, which would profoundly change the character of the problem as the polynomials would cease to be homogeneous and of the same degree. \\ 

Prior to this, similar results concerning the representation of definite forms had only been available in the quadratic case, where a wider range of methods is available. In fact, an asymptotic formula for the number $N(A;B)$ of identical representations of a quadratic form $B \in \Mat(\mathbb Z,m)$ by a positive definite quadratic form $A \in \Mat(\mathbb Z, s)$ has already been known since Raghavan's work of the late 1950s \cite{raghavan}. We may assume without loss of generality that the matrices $A$ and $B$ are non-singular. It follows from the pseudo-diagonality of $B$ that $|b_{i,j}| \le b_{j}$ whenever $b_i<b_j$, so $B$ is essentially Minkowski reduced. Raghavan's result states that for all matrices $A$ and $B$ as above with $s \ge 2m+3$ and $b_i \gg (\det B)^{1/m}$ for all $1 \le i \le m$ there exists some positive constant $\delta$ such that
\begin{align}\label{asymp2}
	N(A;B) = c_{s,m}(\det A)^{-m/2} (\det B)^{\frac{s-m-1}{2}} \prod_{p \text{ prime}}  \chi_{p}(A;B) + O\left(\det B^{\frac{s-m-1}{2} - \delta}\right) \hspace{-0.5mm},
\end{align}
where 
\begin{align}\label{c}
	c_{s,m}= (\sqrt{\pi})^{ms-\frac{m(m+1)}{2}}\prod_{j=m+1}^s \Gamma\left(\frac{j-m}{2}\right)^{\hspace{-1mm} -1}
\end{align}
and 
\begin{align*}
    \chi_{p}(A;B) &= \lim_{l \rightarrow \infty}(p^l)^{\frac{m(m+1)}{2}-ms} \card \left\{X \mmod{p^l} \!: X^t A X \equiv B \pmod{p^l} \right\}\hspace{-.5mm}.
\end{align*}
More recently, Dietmann and Harvey \cite{rd-harvey} showed via the circle method that \eqref{asymp2} is true whenever
\begin{align*}
		s >  2\left(\frac{m(m+1)}{2}+1\right) \left(\frac{m(m+1)}{2}+ \max_{1 \le j \le m}\sum_{i=1}^m \frac{\log (b_{i}/b_{j})}{\log b_{j}}\right) \hspace{-.5mm}.
\end{align*}
Whilst this condition is vastly more restrictive on the number of variables than that of Raghavan, it dispenses with the latter's strict dependence on the relative sizes of the diagonal entries of $B$. However, Theorem \ref{thm:gen} allows us to strengthen their result.

\begin{cor}\label{cor:rep2}
	Suppose $A$ and $B$ are as above, with 
	\begin{align*}
		s>   (m + 1) \max_{1 \le j \le m} \sum_{i=1}^m \frac{\log b_i}{\log b_j}.
	\end{align*}
	Then \eqref{asymp2} is satisfied.
\end{cor}
One observes that, while our bound is still far from Raghavan's, it supersedes that of Dietmann and Harvey by a factor $m^2$ whilst retaining the flexibility with respect to the diagonal entries of $B$. This improvement has been made possible by the advances made in \cite{FRF} regarding the multidimensional version of Birch's theorem.\\

In the case $\psi=0$ Equation \eqref{rep} describes an $m$-dimensional linear space on the hypersurface defined by $F=0$. Indeed, since linear spaces over the integers can be interpreted as lower-dimensional sublattices of $\mathbb Z^s$, our work in \cite{FRF} shows that the number of $m$-dimensional lattices $X \in \mathbb Z^{s \times m}$ with generators of height at most $P$ on which $F$ vanishes identically grows like $P^{ms-rd}$, provided \eqref{FRF-cond} is satisfied and the problem is everywhere locally soluble. 
We may now generalise this question and ask for the number of $m$-dimensional lattices $X$ contained in a non-trivial sublattice of $\mathbb Z^{s \times m}$. Thus for a given matrix $C \in \Mat(\mathbb Z, m)$ we are interested in the number of lattices $X \in \mathbb Z^{s \times m} C$ with generators of length at most $P$ on which $F$ vanishes identically. Such questions arise naturally when one tries to implement an inclusion-exclusion argument on the set of lattices in order to restrict the count to primitive linear spaces. 

\begin{thm}\label{thm:lattice}
    Let $d \ge 2$ and $m \ge 2$ be positive integers, and let $P$ be large. Further, let $C \in \mathbb Z^{m \times m}$ be a non-singular matrix whose Smith normal form is given by $\mathrm{diag}(\gamma_1,\dots,\gamma_m)$. Then, provided that
    \begin{align*}
        s - \dim \sing F>2^{d-1} \max \left\lbrace 2r(d-1), rd \frac{ \sum_{i=1}^m \log (P/\gamma_i) }{ m \log (P/\gamma_{\mathrm{max}})} \right\rbrace\hspace{-.8mm},
    \end{align*}
    the number $N_C(P)$ of points $(\B x_1, \ldots, \B x_m) \in \mathbb Z^{s \times m} C$ of height at most $P$ solving 
    \begin{align}\label{lattice}
    		F(\B x_1 t_t + \dots +\B x_m t_m) =0
    \end{align} 
    identically in $t_1, \ldots, t_m$ is given by
    \begin{align*}
        N_C(P)=\left(\frac{P^{m}}{\det C }\right)^{\hspace{-1mm} s}P^{-rd} \chi_{\infty} \prod_{p \;\mathrm{ prime}} \chi_p + o\left(\left(\frac{P^{m}}{\det C }\right)^{\hspace{-1mm} s}P^{-rd}\right)\hspace{-0.5mm},
    \end{align*}
    and $\chi_p$ and $\chi_{\infty}$ characterise the local solution densities of the variety defined by $F=0$ and are independent of $C$.
\end{thm}

Observe that in the case $m=1$ Theorem~\ref{thm:lattice} reduces to counting solutions $\B x \in (d \mathbb Z)^s$ of height $P$ solving $F(\B x)=0$, which by homogeneity is equivalent to counting $\B x \in \mathbb Z^s$ of height $P/d$. In the higher-dimensional setting, however, the situation is more complicated, but Theorem~\ref{thm:lattice} shows that nonetheless a similar argument can be made even in this case. This situation is generalised for higher values of $m$ here. Of course, a weakness of Theorem \ref{thm:lattice} is that it does not account for unimodular coordinate transforms. This is an issue we hope to address in future work.

Meanwhile, in the case $C=\Id$ the result in Theorem \ref{thm:lattice} as well as Theorem \ref{thm:gen} improves on our former result in \cite{FRF}, replacing the condition $s-\dim \sing F > 3 \cdot 2^{d-1}(d-1)(r+1)$ by the milder $s-\dim \sing F > 2^{d}r(d-1)$. This improvement carries over to the case of systems of $R$ forms, where the respective bound on the number of variables is given by $s-\dim \sing \B F > 2^{d-1}(d-1)Rr(R+1)$.\\

I would like to express my gratitude towards my PhD supervisor Trevor Wooley for his keen insight and constant encouragement, and to my examiners Tim Browning and Rainer Dietmann for valuable comments. This work is based on the author's PhD thesis.

\section{Setup}\label{p:setup}

\textbf{Notation.} Although we believe our notation to be largely self-explanatory, we would like to point out a few items that will be used recurringly.

We will use the Vinogradov and Landau symbols throughout. Whenever the letter $\eps$ occurs, the respective statement is true for all $\eps>0$. We will therefore not trace the particular `value' of each $\eps$, which can consequently change from statement to statement. Furthermore, whenever we write $\sum_{n=a}^b f(n)$ with possibly non-integral boundaries $a$ or $b$, the sum is to be understood to mean $\sum_{a \le n \le b} f(n)$. Occasionally we will write $\sum_{x \ll P}f(x)$, which should be interpreted as $\sum_{-c_1P \le x \le c_2P} f(x) $ with suitable absolute constants $c_1, c_2>0.$

We will abuse vector notation extensively, so any statement involving vectors should be read entry-wise. In this vein we will write $|\B x| \le P$ to mean $|x_i| \le P$ for all entries $i$. Similarly, $(a,\B x)$ denotes the greatest common divisor $(a, x_1, \ldots, x_n)$ of $a$ and all entries of $\B x$. We are confident that no misunderstandings will arise if all similar statements are read in a like manner.

Finally, we will occasionally write $P_{\mathrm{min}}=\min_{i}P_i$, and similarly $\gamma_{\mathrm{max}} = \max_{i} \gamma_i$. \\

Before embarking on the proof of our results, it is useful to show how they are connected with one another. To enhance clarity, we reformulate Theorem~\ref{thm:gen} as follows.

\begin{thm}\label{thm:rep}
    Let $F \in \mathbb Z[x_1,\dots,x_s]$ and $\psi \in \mathbb Z[t_1,\dots,t_m]$ be forms of degree $d \ge 2$, where $m \ge 2$, and let $P_1 \le \dots \le P_m$ be large.  Furthermore, suppose that 
    \begin{align*}
        s - \dim \sing F>2^{d-1} \max \left\lbrace 2r(d-1), rd \left(\frac{ \sum_{i=1}^m \log P_i}{m \log P_1}\right)\right\rbrace\hspace{-.8mm}.
    \end{align*}
     Then the number $N_{\psi}(P_1,\dots,P_m)$ of points $\B x_i \in [-P_i, P_i]^s \cap \mathbb Z^s$ $(1 \le i \le m)$ solving \eqref{rep} identically in $t_1, \ldots, t_m$ is given by
    \begin{equation*}
        N_{\psi}(P_1, \dots,P_m) = \Big(\prod_{i=1}^m P_i \Big) ^{s-rd/m} \chi_{\infty}(F; \psi) \prod_{p \;\mathrm{ prime}} \chi_p(F; \psi) + o\left(\Big(\prod_{i=1}^m P_i \Big) ^{s-rd/m}\right)\hspace{-0.5mm},
    \end{equation*}
    where the factors are given by 
    \begin{align*}
      \chi_{\infty}(F; \psi) = \vol_{\infty} \left\{|{\bm{\xi}_i}|_\infty \le 1  \; (1 \le i \le m): F(\bm{\xi}_1 t_1 + \dots + \bm{\xi}_m t_m) = \psi(\textstyle{\frac{t_1}{P_1},\dots,\frac{t_m}{P_m}}) \right\} 
    \end{align*}
    and 
    \begin{align*}
      \chi_p(F; \psi) = \vol_{p}\{|{\bm{\xi}_i}|_p \le 1  \; (1 \le i \le m): F(\bm{\xi}_1 t_1 + \dots + \bm{\xi}_m t_m) = \psi(t_1,\dots,t_m) \}
    \end{align*}
    and denote the $(ms-r)$-dimensional volume of the normalised set of solutions in the real and $p$-adic unit cubes, respectively. 
\end{thm}

Theorem~\ref{thm:gen} follows from here by setting $P_i = n_i^{1/d}$. With this choice, the dependence on $P_1,\dots,P_m$ of the real solution density amounts to a re-normalisation of the target form to unit length, where $\psi$ is replaced by
\begin{align}\label{psi-til}
    \widetilde\psi (t_1, \dots, t_m) =  \sum_{\B j \in J} \tilde n_{\B j} t_{j_1} \cdot \ldots \cdot t_{j_d}
\end{align}
with coefficients $\tilde n_{\B j}= n_{\B j}(n_{j_1} \cdot \ldots \cdot n_{j_d})^{-1/d}$ which satisfy $\tilde n_{\B j} \le 1$ for all $\B j \in J$ whenever $\psi$ is pseudo-diagonal.
Thus, while the size of the main term is determined by the absolute dimensions of $\psi$ as defined by $\langle \psi \rangle$, the real solution density provides a correction factor by tracking the intrinsic deformity of the body defined by $\psi$ that is preserved after its renormalisation to unit length.

Similarly, Corollary~\ref{cor:rep2} follows from specialising $d=2$, so that the form $B$ is described by an $(m \times m)$-matrix. In the quadratic case the question of whether or not a positive definite form is pseudo-diagonal is easy to answer.

\begin{lem}\label{pseudo-diag}
	Every positive definite quadratic form is pseudo-diagonal. 
\end{lem}
\begin{proof}
	Obviously every positive definite matrix has non-negative diagonal entries, for if $\B e_i$ denotes the $i$-th unit vector, one has $b_i = \B e_i^t B \B e_i \ge 0$. 
	Suppose now that $B$ is a symmetric matrix that is not pseudo-diagonal, so for some $i, j$ one has $|b_{ij}| > \sqrt{b_i b_j} $. Let us first consider the case $b_{ij} > \sqrt{b_i b_j}$, then for every choice of $\lambda, \mu>0$ one has
	\begin{align*}
		(\lambda \B e_i - \mu \B e_j)^t B(\lambda \B e_i - \mu \B e_j) &= \lambda^2 b_i - 2 \lambda \mu b_{ij} + \mu^2 b_j \\
		&< \lambda^2 b_i - 2 \lambda \mu \sqrt{b_ib_j} + \mu^2 b_j= (\lambda \sqrt{b_i} - \mu \sqrt{b_j})^2.
	\end{align*}
	Choosing $\lambda=\sqrt{b_j}$ and $\mu = \sqrt{b_i}$ delivers a contradiction to the assumption that $B$ is positive definite. On the other hand, if $b_{ij} < -\sqrt{b_ib_j}$ we consider the expression 
	\begin{align*}
		(\lambda \B e_i + \mu \B e_j)^t B(\lambda \B e_i + \mu \B e_j) &= \lambda^2 b_i + 2 \lambda \mu b_{ij} + \mu^2 b_j \\
		&< \lambda^2 b_i - 2 \lambda \mu \sqrt{b_ib_j} + \mu^2 b_j = (\lambda \sqrt{b_i} - \mu \sqrt{b_j})^2,
	\end{align*}
	where again we assumed $\lambda, \mu >0$, and the remainder of the argument follows as above. This shows the statement.
\end{proof}
It follows that for any positive definite matrix one has $\det B \asymp b_1 \cdot \ldots \cdot b_m$, and in fact, whilst from Theorem~\ref{thm:rep} we obtain  
\begin{align*}
	\chi_{\infty}(A;B) = \int_{\mathbb R^{m(m+1)/2}} \int_{[-1,1]^{ms}} e \left( \beta_{ij} \bigg(\bm{\xi}_i^T A \bm{\xi}_j - \frac{b_{ij}}{\sqrt{b_i b_j}}  \bigg) \right) \D  \bm{\xi}_1 \cdots \D  \bm{\xi}_m  \D  \bb,
\end{align*}
(see equation \eqref{chi_inf} below), it follows from \cite[\S 6]{rd-harvey} that indeed 
\begin{align*}
	(b_1 \cdot \ldots \cdot b_m)^{(s-m-1)/2}\chi_{\infty}(A;B) &= (\det A)^{-m/2}(\det B)^{(s-m-1)/2} c_{s,m},
\end{align*}
where $c_{s,m}$ is as given in \eqref{c}. \\

Also the third result given in the introduction is essentially a special case of Theorem~\ref{thm:rep} corresponding to the zero polynomial $\psi=0$. First, observe that Theorem~\ref{thm:lattice} depends only on the Smith normal form of $C$. This can be seen as follows. Suppose $U,V \in \SL_m(\mathbb Z)$ are such that $UCV$ is diagonal. By interpreting the vectors $\B x_1, \ldots, \B x_m$ as a matrix $X$, we can view the function $N_C(P)$ as counting matrices $X \in \mathbb Z^{m \times s} C$ of height at most $P$ for which $F(X \B t)= 0$ is true identically in $\B t$. This implies, however, that we may replace $\B t$ by $V\B t$ in the statement and equivalently demand that $F(X V\B t)= 0$ identically in $\B t$. On the other hand, the condition $X \in \mathbb Z^{m \times s} C$ can be written as $X=YC$ with $Y \in \mathbb Z^{m \times s}$, so $N_C(P)$ counts the matrices $Y$ for which $YC$ is of height at most $P$ and $F(YC \B t)= 0$ identically in $\B t$. Since $U$ is unimodular, we may equivalently set $Y=ZU$ and count matrices $Z\in \mathbb Z^{m \times s}$ such that $ZUC$ is of height at most $P$ and $F(ZUC \B t)=0$ identically in $\B t \in \mathbb Z^{m}$.
It follows that we may assume without loss of generality that $C$ is of the form $C = \mathrm{diag}(\gamma_1, \ldots, \gamma_m)$, and the condition
\begin{align*}
    (\B x_1, \ldots, \B x_m)  \in \mathbb Z^{s \times m} C \cap [-P,P]^{s \times m}
\end{align*}
translates into the simpler
\begin{align*}
    \B x_i  \in \gamma_i \mathbb Z^{s} \cap [-P,P]^{s} \qquad (1 \le i \le m).
\end{align*}
It thus remains to relate the counting function
\begin{align*}
 N_C(P) = \card \{ \B x_i \in \gamma_i \mathbb Z^s \cap [-P,P]^s \; (1 \le i \le m): F(\B x_1 t_1 + \dots + \B x_m t_m) = 0 \}
\end{align*}
to the more familiar
\begin{align*}
 N(P/\gamma_1,\dots,P/\gamma_m) = \card \{ |\B x_i| \le P/ \gamma_i\; (1 \le i \le m): F(\B x_1 t_1 + \dots + \B x_m t_m) = 0\}
\end{align*}
which is addressed in Theorem \ref{thm:gen}. 

After expanding, the form $F$ may be written as 
\begin{equation*}
    F\left(t_1 \B x_1+ \ldots + t_m \B x_m\right)
    =\sum_{ \B j \in J} A(\B j) t_{j_1}t_{j_2}\cdot \ldots \cdot t_{j_d} \Phi(\B x_{j_1}, \B x_{j_2},\ldots, \B x_{j_d}),
\end{equation*}
where we use the notation introduced in \cite{FRF} in writing $\Phi$ for the symmetric $d$-linear form associated to $F$ and $A(\B j)$ for the combinatorial factors that take into account the multiplicity of each term. Furthermore, write $\ol x = (\B x_1, \ldots, \B x_m)$ and $\ba=(\alpha_{\B j})_{\B j \in J}$, and let 
\begin{equation*}
    \Phi_{\B j}(\ol x) =  A(\B j)\Phi(\B x_{j_1}, \dots \B x_{j_d}) \qquad (\B j \in J)
\end{equation*}
and
\begin{align}\label{def-F}
    \F F\big(\B x_1, \ldots, \B x_m; \ba\big)= \sum_{ \B j \in J} \alpha_{\B j}  \Phi_{\B j}(\ol x) .
\end{align}
In the context of Theorem~\ref{thm:rep} the variables $\B x_i$ lie in intervals $[-P_i,P_i]$, the Cartesian product of which which we denote by $\C P$. Notice that this can be transformed into the language of Theorem~\ref{thm:lattice} by by setting $P_m = P$ and $\gamma_i = P_m/P_i$ for all $i$. In this notation we have
\begin{equation*}
    \card \C P = \prod_{i=1}^m P_i = \prod_{i=1}^m \frac{P}{\gamma_i} = \frac{P^m}{\det C }.
\end{equation*}
Thus classical orthogonality relations imply that the number of solutions to \eqref{lattice} with $\B x_i \le P_i$ is given by
\begin{equation*}
    N_\psi(\C P)=  \int_{[0,1)^r} T(\ba; \C P) e(- \B n \cdot \ba)\D  \ba, 
\end{equation*}
where the exponential sum is defined as
\begin{equation*}
    T(\ba; \C P) = \sum_{\ol x \in \C P^s} e(\F F(\ol  x; \ba)).
\end{equation*}

In the case of Theorem~\ref{thm:lattice}, on the other hand, the counting function takes the shape
\begin{align*}
    N_C(P)&=  \int_{[0,1)^r}\sum_{\ol x C \le P} e\left( \F F\left(\ol  xC; \ba\right)\right) \D  \ba.
\end{align*}
Since we may assume $C$ to be diagonal, the variables $\B x_i$ are multiples of the diagonal entries $\gamma_i$ of $C$, and by homogeneity we may write
\begin{align*}
    \F F\big(\gamma_1 \B x_1, \ldots, \gamma_m \B x_m; \ba\big)&= \sum_{ \B j \in J} \alpha_{\B j}   A(\B j)\Phi(\gamma_{j_1}\B x_{j_1}, \ldots, \gamma_{j_d}\B x_{j_d}) \\
    &=\sum_{ \B j \in J} \alpha_{\B j} \hgam A(\B j)\Phi(\B x_{j_1}, \ldots, \B x_{j_d}) ,
\end{align*}
where we introduced the notation $\hgam$ for the product $\gamma_{j_1} \cdot \ldots \cdot \gamma_{j_d}$. Absorbing the factors $\hgam$ into the coefficients $\alpha_{\B j}$, we see that the number of solutions is given by
\begin{align*}
    N_C(P) =\int_{[0,1)^r}\sum_{\ol x \in \C P^s} e\left( \sum_{ \B j \in J} \alpha_{\B j} \hgam \Phi_{\B j}(\ol x) \right) \D  \ba =\Big(\prod_{\B j \in J} \hgam\Big)^{-1}  N(\C P).
\end{align*}
The product in the last expression is symmetric in the $\gamma_i$ and has altogether $rd$ factors, so its value is
\begin{align}\label{prod-hhgam}
    \prod_{\B j \in J}\hgam  = (\gamma_1 \cdot \ldots \cdot \gamma_m)^{rd/m} =  (\det C)^{rd/m}.
\end{align}
It is therefore the counting function $N_{\psi}(\C P)$ considered in Theorem~\ref{thm:rep} towards which we will direct our attention.

\section{The tripartite Weyl inequality}

The first step is to establish an inequality of Weyl type. Although the greater picture of this is by now fairly standard, the different ranges of the $\B x_i$ create some technical complications which need to be attended to with due care.

We define the discrete difference operator by its action on the form $\F F(\ol x; \ba)$ as
\begin{equation*}
    \Delta_{i, \B h} \F F(\ol  x; \ba) = \F F(\B x_1, \ldots, \B x_{i} + \B{h}, \ldots, \B x_m; \ba) - \F F(\B x_1, \ldots, \B x_{i}, \ldots, \B x_m;\ba),
\end{equation*}
and write for brevity
\begin{equation*}
    \Delta^{(k)}_{\B h_{\B j}} = \Delta_{k, \B h_{j_k}} \cdots \Delta_{1, \B h_{j_1}}.
\end{equation*}
This allows us to formulate our first Weyl differencing lemma.

\begin{lem}\label{lem3:weyl1}
    Let $1 \le k \le d-1$ and fix $\B j \in J$. We have
    \begin{equation*}
        | T(\ba; \C P)|^{2^k} \ll \left(\frac{P^m}{\det C }\right)^{\hspace{-1mm} (2^k-1)s} \Big(\prod_{l=1}^k P_{j_l}^{-s}\Big) \dsum{\B h_l \ll P_{j_l}}{1 \le l \le k} \sum_{\ol x }  e\left(\Delta^{(k)}_{\B h_{\B j}} \F F(\ol x; \ba)\right)\hspace{-0.5mm},
    \end{equation*}
    and the sum over $\ol x$ is over a suitable box contained in $ \C P^s$.
\end{lem}

\begin{proof}
    The proof is, as usual, by induction.
     The case $k=1$ follows by Cauchy--Schwarz via
    \begin{align*}
        | T(\ba; \C P)|^2 &\ll \Bigg( \dsum{|\B x_i| \le P_i}{i \neq j_1} 1\Bigg) \dsum{|\B x_i| \le P_i}{i \neq j_1}  \Bigg|\sum_{|\B x_{j_1}| \le P_{j_1}} e(\F F(\ol x; \ba)) \Bigg|^2 \\
        &\ll  \left(\frac{P^m}{\det C }\right)^{\hspace{-1mm} s}   P_{j_1}^{-s} \sum_{|\B h_1| \le P_{j_1}} \sum_{\ol x } e \left( \Delta_{j_1, \B h_1} \F F(\ol x; \ba)  \right)\hspace{-0.5mm}.
    \end{align*}
    Note that the final summation of $\B x_{j_1}$ is over the set 
    \begin{align*}
	\{ \B x_{j_1} \in \mathbb Z^s: |\B x_{j_1}| \le P_{j_1}, |\B x_{j_1}+\B h_1| \le P_{j_1}\},
    \end{align*}
    which is again a box contained in $[-P_{j_1}, P_{j_1}]^s$.
    The induction step is similar. By another application of Cauchy--Schwarz one has
    \begin{align*}
         &\Bigg|\dsum{\B h_l \ll P_{j_l}}{1 \le l \le k} \sum_{\ol  x }  e\left(\Delta^{(k)}_{\B h_{\B j}} \F F(\ol x; \ba)\right)\Bigg|^2 \\
         &\quad \ll \Bigg( \dsum{\B h_l \ll P_{j_l}}{1 \le l \le k} \dsum{\B x_i \ll P_i}{i \neq j_{k+1}} 1 \Bigg) \dsum{\B h_l \ll P_{j_l}}{1 \le l \le k}\dsum{\B x_i \ll P_i}{i \neq j_{k+1}} \Bigg|\sum_{\B x_{j_{k+1}} \ll P_{j_{k+1}}}e\left(\Delta^{(k)}_{\B h_{\B j}} \F F(\ol x; \ba)\right)  \Bigg|^2 \\
         &\quad \ll \left(\frac{P^m}{\det C }\right)^{\hspace{-1mm} s} P_{j_{k+1}}^{-s} \Big(\prod_{l=1}^k P_{j_l}^s\Big)  \dsum{\B h_l \ll P_{j_l}}{1 \le l \le k+1} \sum_{\ol  x} e\left(\Delta^{(k+1)}_{\B h_{\B j}} \F F(\ol x; \ba)\right),\hspace{-1mm}
    \end{align*}
    and hence
    \begin{align*}
        | T(\ba; \C P)|^{2^{k+1}} & \ll \left(\frac{P^m}{\det C }\right)^{\hspace{-1mm} (2^{k+1}-2)s} \Big(\prod_{l=1}^k P_{j_l}^{-2s}\Big) \Bigg| \dsum{\B h_l \ll P_{j_l}}{1 \le l \le k} \sum_{\ol x}  e\left(\Delta^{(k)}_{\B h_{\B j}} \F F(\ol x; \ba)\right)\Bigg|^2 \\
        &\ll \left(\frac{P^m}{\det C }\right)^{\hspace{-1mm} (2^{k+1}-1)s} \Big(\prod_{l=1}^{k+1} P_{j_l}^{-s}\Big)  \dsum{\B h_l \ll P_{j_l}}{1 \le l \le k+1} \sum_{\ol x } e\left(\Delta^{(k+1)}_{\B h_{\B j}} \F F(\ol x; \ba)\right)\hspace{-0.5mm}.
    \end{align*}
    This completes the proof.
\end{proof}

As is usual in Weyl differencing arguments, we notice that the differencing procedure gradually reduces the degree while preserving the structure of the system. In our case, this means that after $d-1$ applications the resulting expression is linear in the variables $\B x_1, \ldots, \B x_m$, and since all forms in the system of equations are instances of the same multilinear form $\Phi$, all of these linear expressions will be of the same shape. We abbreviate $\hs$ for the $(d-1)$-tuple $(\B h_1, \ldots, \B h_{d-1})$, and define the $(d-1)$-linear forms $B_n$, $1 \le n \le s$, via the relation
\begin{equation*}
    \Phi(\B x, \hs) = \sum_{n=1}^s B_n(\hs) x_{n},
\end{equation*}
so that one has 
\begin{align}\label{def-B}
    \Delta^{(d-1)}_{\B h_{\B j}} \F F(\ol x; \ba)&= \sum_{j_d=1}^m\sum_{n=1}^s\alpha_{\B j} M(\B j) B_n(\hs)x_{j_d,n} + R(\hs)
\end{align}
for some combinatorial factors $M(\B j)$ and some function $R(\hs)$ collecting the terms independent of $\ol x$. Lemma~\ref{lem3:weyl1} yields now
\begin{equation*}
		| T(\ba; \C P)|^{2^{d-1}} \ll \left(\frac{P^m}{\det C }\right)^{\hspace{-1mm} (2^{d-1}-1)s} \Big(\prod_{l=1}^{d-1} P_{j_l}^{-s}\Big) \dsum{\B h_l \ll P_{j_l}}{1 \le l \le d-1} \sum_{\ol x}  e\left(\Delta^{(d-1)}_{\B h_{\B j}} \F F(\ol x; \ba)\right)\hspace{-0.5mm},
\end{equation*}
and by \eqref{def-B} we have
\begin{align*}
    \sum_{\ol x}  e\left(\Delta^{(d-1)}_{\B h_{\B j}} \F F(\ol x; \ba)\right)&\ll\left(\frac{P^m}{\det C }\right)^{\hspace{-1mm} s} P_{j_d}^{-s}  \left| \sum_{|\B x_{j_d}| \ll P_{j_d}}  e\left(M(\B j) \alpha_{\B j} \Phi(\B x_{j_d}, \hs)\right)\right| \\
		&\ll \left(\frac{P^m}{\det C }\right)^{\hspace{-1mm} s} P_{j_d}^{-s} \prod_{n=1}^s  \min\left(P_{j_d},  \| M(\B j) \alpha_{\B j} B_n(\hs) \|^{-1} \right) \hspace{-0.5mm}.
\end{align*}
It follows that if we define 
\begin{equation*}
    \Upsilon(\B j) = \dsum{\B h_l \ll P_{j_l}}{1 \le l \le d-1} \prod_{n=1}^s \min\left(P_{j_d},  \| M(\B j) \alpha_{\B j}  B_n(\hs) \|^{-1} \right)\hspace{-0.5mm},
\end{equation*}
then we can bound the exponential sum as
\begin{align}\label{weyl2}
		| T(\ba; \C P)|^{2^{d-1}} \ll \left(\frac{P^m}{\det C }\right)^{\hspace{-1mm} 2^{d-1}s} \Big(\prod_{k=1}^d P_{j_k}^{-s}\Big)  \Upsilon(\B j).
\end{align}

We want to conclude from \eqref{weyl2} that either $ T(\ba; \C P)$ can be bounded non-trivially, or we have a good rational approximation to the vector $\ba$, or else the form $F$ has been highly singular from the beginning. 
Let $N_{\B j}(X_1, \ldots, X_{d-1}; Y)$ denote the cardinality of the set
\begin{align}\label{def-N}
    \Big\{|\B h_{k}| \le X_k \quad (1 \le k \le d-1)  \! :  \|M(\B j) \alpha_{\B j} B_n(\hs) \|^{-1} > Y \quad (1 \le n \le s)\Big\}.
\end{align}
We show that when $T(\ba; \C P)$ is large, then $N_{\B j}(P_{j_1}, \ldots, P_{j_{d-1}}; P_{j_d})$ is also large.
\begin{lem}\label{lem3:pre-geom}
    Suppose
    \begin{equation*}
        | T(\ba; \C P)| \gg \left(\frac{P^m}{\det C }\right)^{\hspace{-1mm} s} P^{-K}
    \end{equation*}
    for some parameter $K>0$. Then for every $\B j \in J$ one has
    \begin{equation*}
        N_{\B j}\left(P_{j_1}, \ldots, P_{j_{d-1}}; P_{j_d}\right) \gg P^{-2^{d-1}K-\eps} \prod_{k=1}^{d-1} P_{j_k}^s .
    \end{equation*}
\end{lem}

\begin{proof}
    Combining the hypothesis of the lemma with \eqref{weyl2}, we have
    \begin{align*}
        \left(\left(\frac{P^m}{\det C }\right)^{\hspace{-1mm} s} P^{-K}\right)^{\hspace{-1mm} 2^{d-1}}  & \ll \left(\frac{P^m}{\det C }\right)^{\hspace{-1mm} 2^{d-1}s} \Big(\prod_{k=1}^d P_{j_k}^{-s}\Big)  \Upsilon(\B j).
    \end{align*}
    Rearranging the terms, one obtains
    \begin{align}\label{ups>sth}
        \Upsilon(\B j) & \gg P^{-2^{d-1}K} \prod_{k=1}^d P_{j_k}^s.
    \end{align}
    Note that, unlike in the situation considered in \cite{FRF}, we obtain distinct estimates for different indices $\B j \in J$.

    Now for a fixed tuple $(\B h_2, \ldots, \B h_{d-1})$ write
    \begin{equation*}
        R(\B h_2, \ldots, \B h_{d-1}) = \card \{  \B h_1 \ll  P_{j_1} \! : \| M(\B j) \alpha_{\B j} B_n(\hs) \| < P_{j_d}^{-1} \quad (1 \le n \le s)\},
    \end{equation*}
    so that
    \begin{equation*}
        \dsum{\B h_l \ll P_{j_l}}{2 \le l \le d-1} R(\B h_2, \ldots, \B h_{d-1}) =N_{\B j}\left(P_{j_1}, \ldots, P_{j_{d-1}}; P_{j_d}\right)\hspace{-0.5mm}.
    \end{equation*}
    Then a familiar pigeonhole argument as in the proof of \cite[Lemma~13.2]{dav} implies that for any $s$-tuple of integers $(r_1, \dots, r_s)$ between $0$ and $P_{j_d}-1$ and any fixed $\B h_2, \dots, \B h_{d-1}$ the number of $\B h_1 \ll P_{j_1}$ satisfying
    \begin{equation*}
        \frac{r_n}{P_{j_d}} < \{M(\B j) \alpha_{\B j}B_{\B j,n}(\hs)\} < \frac{r_n+1}{P_{j_d}} \quad (1 \le n \le s)
    \end{equation*}
    is at most $ R(\B h_2, \ldots, \B h_{d-1})$, and thus
    \begin{align*}
         \Upsilon(\B j) &\ll \dsum{\B h_l \ll P_{j_l}}{2 \le l \le d-1} \prod_{n=1}^s \sum_{r_n=1}^{P_{j_d}-1} \min \left\{P_{j_d}, \frac{P_{j_d}}{r_n}, \frac{P_{j_d}}{r_n+1} \right\}\\
         &\ll (P_{j_d} \log P_{j_d})^s \dsum{\B h_l \ll P_{j_l}}{2 \le l \le d-1} R(\B h_2, \ldots, \B h_{d-1})\\
         & \ll (P_{j_d} \log P_{j_d})^s  N_{\B j}\left(P_{j_1}, \ldots, P_{j_{d-1}}; P_{j_d}\right)\hspace{-0.5mm}.
    \end{align*}
    Inserting this into \eqref{ups>sth} gives the desired result.
\end{proof}

We will need the following standard lemma.

\begin{lem}\label{lem3:dav}
    Let $L_1, \ldots, L_n$ be linear forms given by
    \begin{equation*}
        L_i = \lambda_{i,1}x_1 + \ldots + \lambda_{i,n}x_n \quad (1 \le i \le n)
    \end{equation*}
    with the additional symmetry that $\lambda_{i,j}=\lambda_{j,i}$. For a parameter $A>1$ let $U(Z)$ denote the number of integer solutions $x_1, \ldots, x_n$ to the system
    \begin{equation*}
        |x_i| < AZ \quad \hbox{and} \quad \| L_i(\B x)\| < Z/A \quad (1 \le i \le n).
    \end{equation*}
    Then for $0 < Z' \le Z \le 1$ we have
    \begin{equation*}
        \frac{U(Z)}{U(Z')} \ll \left( \frac{Z}{Z'}\right)^{\hspace{-1mm}n}\hspace{-0.5mm}.
    \end{equation*}
\end{lem}
\begin{proof}
    This is \cite[Lemma~12.6]{dav}.
\end{proof}

The strategy is now to apply Lemma~\ref{lem3:dav} to each of the variables $\B h_k$ in such a way that $AZ=P_{j_k}$ and $AZ' = P^{\theta}$ for some small exponent $\theta$, so that in the further course of the argument we can assume the variables to lie in small boxes which are then independent of $C$. However, this is legitimate only in the case when $Z' \le Z$, so we need $(AZ')/(AZ) = P^\theta/P_{j_k} \le 1$ for all $k$. This condition amounts to $P^\theta \le P/\gamma_{\mathrm{max}}$ or, taking logarithms,
\begin{align*}
    \theta \le 1- \frac{\log \gamma_{\mathrm{max}}}{\log P}.
\end{align*}
For simplicity we write
\begin{align}\label{eta-def}
    \eta =\log \gamma_{\mathrm{max}}/\log P.
\end{align}
Notice that in the case when $C$ is the identity matrix we have $\eta=0$ and therefore $\theta \le 1$ as usual. 

\begin{lem}\label{lem3:geom}
    Suppose that $0 < \theta \le 1-\eta$, where $\eta$ is as in \eqref{eta-def}. Then under the hypothesis of Lemma~\ref{lem3:pre-geom} one has
    \begin{equation*}
        N_{\B j}\left(P^{\theta}, \ldots, P^{\theta}; \frac{P^{d-(d-1)\theta}}{\hgam}\right) \gg P^{(d-1)s\theta-2^{d-1}K-\eps}.
    \end{equation*}
\end{lem}

\begin{proof}
    This follows from Lemma~\ref{lem3:dav} by what is essentially a standard argument. We may assume without loss of generality that $\gamma_i \ge \gamma_j$ for $i<j$ and that the components of every multi-index $\B j \in J$ are arranged in ascending order, so that $\gamma_{j_1} \ge \gamma_{j_2} \ge \dots \ge \gamma_{j_d}$ for all $\B j \in J$. For fixed $\B h_1, \ldots \B h_{k-1}, \B h_{k+1}, \ldots , \B h_{d-1}$ and fixed $A>1$ let $U_k(Z)$ denote the number of $\B h_k < A Z$ such that $\|M(\B j) \alpha_{\B j} B_n(\hs) \| <Z/A$. We will take
    \begin{align*}
        A_1 &= \frac{P}{\sqrt{\gamma_{j_1}\gamma_{j_d}}}, &  Z_1 &= \sqrt{\frac{\gamma_{j_d}}{\gamma_{j_1}}}, &   Z'_1 &= \frac{\sqrt{\gamma_{j_d}\gamma_{j_1}}}{P^{1-\theta}},
    \end{align*}
    and then recursively
    \begin{align*}
        A_k &= A_{k-1}\sqrt{\frac{P^{1-\theta}}{\gamma_{j_k}}}, &  Z_k &= \frac{\gamma_{j_{k-1}}Z_{k-1}}{\sqrt{P^{1-\theta}\gamma_{j_k}}}, &  Z'_k &= \frac{\gamma_{j_k} Z_k}{P^{1-\theta}},
    \end{align*}
    so that the relations
    \begin{align*}
        A_kZ_k&=P/\gamma_{j_k}, &  Z_k/A_k &=P^{-k+(k-1)\theta}\gamma_{j_1} \cdot \ldots \cdot \gamma_{j_{k-1}}\gamma_{j_d},\\
        A_k Z'_k&=P^{\theta}, &  Z'_{k}/A_{k} &= P^{-(k+1)+k\theta}\gamma_{j_1} \cdot \ldots \cdot \gamma_{j_{k}}\gamma_{j_d},\\
        Z_k/Z_k' & = P^{1-\theta}/\gamma_{j_k} &&
    \end{align*}
    are satisfied for each $1 \le k \le d-1$. Observe that our hypothesis on $\theta$ ensures that $Z_k \ge Z_k'$. Furthermore, one easily confirms that, according to our assumptions on the relative sizes of the $\gamma_{j_k}$, one has   
    $$Z_k = \sqrt{\frac{\gamma_{j_d}}{\gamma_{j_k}}} \prod_{l=1}^{k-1} \sqrt \frac{\gamma_{j_l}}{P^{1-\theta}} \le 1.$$
    It follows that Lemma~\ref{lem3:dav} is applicable and yields
    \begin{equation}\label{shortening}
        U_k(Z_k) \ll \left( \frac{Z_k}{Z_k'}\right)^{\hspace{-1mm} s}U_k\left(Z_k'\right) \ll (P^{1-\theta}/\gamma_{j_k})^{s} U_k\left(Z_k'\right)\hspace{-.5mm}.
    \end{equation}
    For a given $k$ between $0$ and $d-1$ consider the quantity
     \begin{align}\label{N(P,P)-est}
          \nu(k)= N_{\B j}\Bigg(\underbrace{P^{\theta}, \ldots, P^{\theta}}_{\text{first $k$ entries}}, \frac{P}{\gamma_{k+1}}, \ldots, \frac{P}{\gamma_{d-1}}; \frac{P^{(k+1)-k\theta}}{\gamma_{j_1}\cdot\ldots\cdot\gamma_{j_k}\gamma_{j_d}}\Bigg),
    \end{align}
   then $\nu(k)$ can be expressed in terms of $U_k(Z_k')$ or $U_{k+1}(Z_{k+1})$ by the relation
    \begin{align}\label{nu-def}
        \nu(k) &= \dsum{\B h_l \ll P^{\theta}}{1\le l \le k-1} \dsum{\B h_l \ll P/\gamma_{j_l}}{k+1\le l \le d-1} U_k\left(Z'_{k}\right)= \dsum{\B h_l \ll P^{\theta}}{1\le l \le k} \dsum{\B h_l \ll P/\gamma_{j_l}}{k+2\le l \le d-1} U_{k+1}\left(Z_{k+1}\right)\hspace{-.5mm}.
    \end{align}
    On combining \eqref{shortening} and \eqref{nu-def} we obtain a recursive relation for the $\nu(k)$ which is given by $\nu(k-1)\ll  (P^{1-\theta}/\gamma_{j_k})^{s} \nu(k)$. It follows that 
    \begin{align*}
        \nu(d-1)\gg P^{-(1-\theta)(d-1)s} \Big(\prod_{k=1}^{d-1}\gamma_{j_k}^{s}\Big)\nu(0),
    \end{align*}
    and on recalling the definition \eqref{N(P,P)-est} of $\nu(k)$ we find
    \begin{align*}
        N_{\B j}\left(P^{\theta}, \ldots, P^{\theta}; \frac{P^{d-(d-1)\theta}}{\hgam}\right)\gg P^{-(d-1)(1-\theta)s}\Big(\prod_{k=1}^{d-1}\gamma_{j_k}^{s}\Big)N_{\B j}\left(P_{j_1}, \ldots, P_{j_{d-1}}; P_{j_d}\right).
		\end{align*}		
		Inserting Lemma~\ref{lem3:pre-geom} now yields the desired result. 
\end{proof}

The content of Lemma~\ref{lem3:geom} is that if the exponential sum is large, the quantities $ M(\B j) \alpha_{\B j} B_n(\hs)$ are simultaneously close to an integer for many choices of $\hs$. This is certainly true if the forms $B_n$ tend to vanish for geometric reasons, and in the other case it implies that one can find genuine (i.e. non-zero) solutions to the diophantine approximation problem that is implicit in \eqref{def-N}. This yields the standard threefold case distinction.

\begin{lem}\label{lem3:ram1}
    Let $0 < \theta \le 1-\eta$ and $k>0$ be parameters, and let $\ba \in [0,1)^{r}$. Then one of the following is true.
    \begin{enumerate}[(A)]
    \item
        The exponential sum $ T(\ba; \C P)$ is bounded by
        \begin{equation*}
            | T(\ba; \C P)| \ll \left(\frac{P^m}{\det C }\right)^{\hspace{-1mm} s} P^{-k\theta}.
        \end{equation*}
    \item
        For every $\B j \in J$ one finds integers $(q_{\B j}, a_{\B j}) $ satisfying
        \begin{align*}
            0<q_{\B j} \ll P^{(d-1)\theta}  \qquad \hbox{and} \qquad       \left|\alpha_{\B j} q_{\B j} - a_{\B j}\right| \ll   P^{-d +(d-1)\theta} \hgam.
        \end{align*}
    \item 
        The number of  $|\B h_l| \le P^{\theta}$ for $1 \le l \le d-1$ that satisfy
        \begin{equation*}
            B_n(\B h_1, \ldots, \B h_{d-1}) = 0 \qquad (1 \le n \le s)
        \end{equation*}
        is asymptotically greater than $(P^\theta)^{(d-1)s - 2^{d-1}k - \eps}$.
    \end{enumerate}
\end{lem}

\begin{proof}
    The proof is similar to that of \cite[Lemma~3.4]{FRF}. Assuming that the estimate in (A) does not hold, Lemma~\ref{lem3:geom} implies that for every $\B j \in J$ we have
    \begin{equation*}
        \big\| M(\B j) \alpha_{\B j} B_n(\hs_{\B j}) \big\| < P^{-d+(d-1)\theta} \hgam \qquad (1 \le n \le s)
    \end{equation*}
    for  at least $P^{(d-1)s\theta-2^{d-1}k\theta - \eps}$ choices of $\hs_{\B j} \le P^{\theta}$. If $B_n(\hs_{\B j})$ is non-zero for some $n$ and some $\hs_{\B j}$ counted by $N_{\B j}(P^{\theta}, \dots, P^{\theta};P^{d-(d-1)\theta}/\hgam)$, we denote its value by $q_{\B j}$. Obviously $q_{\B j} \ll P^{(d-1)\theta}$, and it follows that we can find an integer $a_{\B j}$ with the property that
    \begin{equation*}
          |\alpha_{\B j} q_{\B j}   -  a_{\B j}|  \ll P^{-d+(d-1)\theta}\hgam.
    \end{equation*}
    This establishes the statement.
\end{proof}

By choosing the number of variables large enough, the singular case can be excluded. This is, however, identical to the treatment in \cite[Lemma~3.5]{FRF}.

\begin{lem}\label{lem3:ram-final}
    Let $\ba \in [0,1)^{r}$ and let $0 < \theta \le 1-\eta$ and $k$ be parameters with
        \begin{equation}\label{s>k}
                s - \dim \sing F > 2^{d-1}k.
        \end{equation}
    Then the alternatives are the following.
    \begin{enumerate} [(A)]
        \item \label{case3:minor}
            The exponential sum $ T(\ba; \C P)$ is bounded by
            \begin{equation*}
                | T(\ba; \C P)| \ll \left(\frac{P^m}{\det C }\right)^{\hspace{-1mm} s} P^{-k\theta+\eps}.
            \end{equation*}
        \item  \label{case3:major}
            For every $\B j \in J$ one finds a pair of coprime integers $(q_{\B j}, a_{\B j}) $ satisfying
            \begin{align*}
                0< q_{\B j} &\ll P^{(d-1)\theta}  \qquad \hbox{and} \qquad   \left|\alpha_{\B j} q_{\B j} - a_{\B j}\right| \ll   P^{-d+(d-1)\theta}\hgam.
            \end{align*}
    \end{enumerate}
\end{lem}

\begin{proof}
    This is essentially Lemma~3.5 of \cite{FRF}. Notice that the singular case in Lemma~\ref{lem3:ram1} is the same as that in Lemma~3.4 of our former work, and in particular does not depend on the matrix $C$ or indeed on linear spaces altogether, so the methods used to derive Lemma~3.5 from Lemma~3.4 of \cite{FRF} are applicable, and the singular case is excluded by our assumption \eqref{s>k}.
\end{proof}

\section{Implementation of the circle method}

Let $c$ be sufficiently large in terms of $C$ and the coefficients of $F$, then we write $\F M(P,\theta)$ for the set of all $\ba \in [0,1)^{r}$ that have a rational approximation satisfying
\begin{align*}
    0 \le a_{\B j} <q_{\B j} \le c P^{(d-1)\theta}   \qquad \hbox{and}  \qquad \left|\alpha_{\B j} q_{\B j} - a_{\B j}\right| \le c P^{-d+(d-1)\theta} \hgam
\end{align*}
for all $\B j \in J$, and $\F m(P,\theta)$ for the complement thereof. Notice that this respects the case distinction of Lemma~\ref{lem3:ram-final}, so there is a non-trivial minor arcs estimate for all $\ba \in \F m(P, \theta)$. In order to keep notation simple, we omit the parameter $P$ whenever there is no danger of confusion.

\begin{lem} \label{lem3:pruning}
    Suppose the parameters $k$ and $\theta$ satisfy
    \begin{align*}
        0 < \theta  < \theta_0 = \min \left\lbrace \frac{d}{2(d-1)}, 1-\eta \right\rbrace
    \end{align*}
    and
    \begin{equation}\label{k-pruning}
         k >  \max \left\lbrace 2r(d-1), rd\frac{\sum_{i=1}^m \log P_i}{m \log P_{\mathrm{min}}} \right\rbrace \hspace{-0.5mm}.
    \end{equation}
    Then there exists a $\delta>0$ such that the minor arcs contribution is bounded by
    \begin{equation*}
        \int_{\F m(\theta)} | T(\ba; \C P)| \D \ba = O\left(  \left(\frac{P^m}{\det C }\right)^{\hspace{-1mm} s-rd/m}P^{-\delta} \right)\hspace{-0.5mm}.
    \end{equation*}
\end{lem}
\begin{proof}
    We follow the proof of \cite[Lemma~4.4]{birch}. As a first step, we note that     
    \begin{align*}
	\left(\frac{P^m}{\det C }\right)^{\hspace{-1mm} -rd/m} = P^{- rd(1-\Omega)},
    \end{align*}
    where $\Omega =  \frac1m \sum_{i=1}^m \frac{\log \gamma_i}{\log P}$. It follows that $P^{-k \theta_0} \ll  \left(P^m/\det C \right)^{-rd/m-\delta}$ as soon as either one has $\theta_0 = \frac{d}{2(d-1)}$ and $k > 2r(d-1)$, or $\theta_0 = 1-\eta$ and 
    \begin{align*}
	k&> rd\frac{1-\Omega}{1-\eta}= rd\frac{m \log P - \sum_{i=1}^m \log \gamma_i}{ m(\log P - \log \gamma_{\mathrm{max}})} = rd\frac{\sum_{i=1}^m \log P_i}{m \log P_{\mathrm{min}}}.
    \end{align*}
    It follows that for some $\delta>0$ the contribution arising from $\F m(\theta_0)$ is given by  
    \begin{equation*}
        \int_{\F m(\theta_0)} | T(\ba; \C P)| \D \ba \ll \left(\frac{P^m}{\det C }\right)^{\hspace{-1mm} s-rd/m}P^{-\delta}.
    \end{equation*}
    
    Given $0 < \theta < \theta_0$, we can find $\delta>0$ and $1 \ge \theta_0 > \theta_1 > \ldots > \theta_M = \theta >0$ satisfying
    \begin{equation} \label{delta}
         (k-2r(d-1))\theta >2\delta \quad \text{and} \quad (\theta_i-\theta_{i+1})k<\delta \quad \text{for all $i$}.
    \end{equation}
    This is always possible with $M=O(1)$.
    Then on writing
    \begin{equation*}
        \F m_i = \F m(\theta_i) \setminus \F m(\theta_{i-1}) = \F M(\theta_{i-1}) \setminus \F M(\theta_{i})
    \end{equation*}
    a straightforward calculation shows that     
		\begin{equation*}
        \vol \F m_i  \le \vol \F M(\theta_{i-1}) \ll \left(\frac{P^m}{\det C }\right)^{\hspace{-1mm} -rd/m}P^{2(d-1)r\theta_{i-1} }.
    \end{equation*}
    Recall that for $\ba \in \F m(\theta_i)$ we are in the situation of case \eqref{case3:minor} in Lemma~\ref{lem3:ram-final}, so the $i$-th minor arcs contribution is bounded by
    \begin{align*}
        \int_{\F m_i} \left| T(\ba; \C P) \right| \D  \ba         &\ll  \vol \F M(\theta_{i-1})  \sup_{\ba \in \F m(\theta_i) } \left| T(\ba; \C P) \right| \\
        &\ll  \left(\frac{P^m}{\det C }\right)^{\hspace{-1mm} -rd/m}P^{2(d-1)r\theta_{i-1}} \left(\frac{P^m}{\det C }\right)^{\hspace{-1mm} s}P^{-k\theta_i}\\
        &\ll  \left(\frac{P^m}{\det C }\right)^{\hspace{-1mm} s-rd/m} P^{-k\theta_i+2r(d-1)\theta_{i-1}}.
    \end{align*}
    By \eqref{delta} the exponent is
    \begin{align*}
        -k\theta_i+2r(d-1)\theta_{i-1} &= k(\theta_{i-1}-\theta_i)-(k-2r(d-1))\theta_{i-1}\\
        &< k(\theta_{i-1}-\theta_i)-(k-2r(d-1))\theta <-\delta
    \end{align*}
    and on summing the $O(1)$ contributions with $1 \le i \le M$ we recover the statement.
\end{proof}

We now have to homogenise the major arcs in order to find a common denominator for the major arcs approximations. However, for our present context it is enough to give only a crude bound here and let $q = \lcm_{\B j \in J} q_{\B j}$, so that trivially $q \ll P^{r(d-1)\theta}$ and $|\alpha_{\B j}q-b_{\B j}|  \ll 	P^{ -d+r(d-1)\theta}\hgam$.

For technical reasons, it is convenient to extend the major arcs slightly and define $\F M'(\theta)$ to be the set of all $\ba = \B a/q + \bb$ contained in the interval $[0,1)^{r}$ that satisfy
\begin{align} \label{major'}
    0 \le \B a < q \le c' P^{r(d-1)\theta } \qquad \hbox{and} \qquad      | \beta_{\B j} | \le c' P^{-d+r(d-1)\theta}\hgam \quad (\B j \in J)
\end{align}
for some suitably large constant $c'$. Henceforth all parameters $\ba, \B a, q, \bb$ will be implicitly understood to satisfy the major arcs inequalities as given in \eqref{major'}.
Let
\begin{align*}
    S_q(\B a) = \sum_{\ol x \mmod q} e \left( \frac {\F F(\ol x; \B a)}{q}  \right) 
\end{align*}
and
\begin{align*}
    v_{\C P}(\bb) = \int_{\ol{\bm{\xi}} \in \C P^s} e \left( \F F(\ol{\bm{\xi}};  \bb) \right) \D  \ol{\bm{\xi}}.
\end{align*}
We can now replace the exponential sum by an expression in terms of the approximation given by $\B a$, $q$ and $\bb$.

\begin{lem} \label{lem3:gen}
    Assume that $\ba  = \B a/q + \bb$, then one has
    \begin{align*}
          T(\ba; \C P)  - q^{-ms} S_q(\B a) v_{\C P}(\bb ) &\ll q^{ms}  + \left(\frac{P^m}{\det C }\right)^{\hspace{-1mm} s}  \left(\sum_{\B j \in J} |\beta_{\B j}|\frac{P^{d}}{\hgam}\right)\frac{q}{P_{\mathrm{min}}}.
    \end{align*}
\end{lem}

\begin{proof}
    We sort the variables $\ol x$ occurring in $T(\ba; \C P)$ in arithmetic progressions modulo $q$ and write $\ol x = q \ol w + \ol z$. Then 
    \begin{align*}
         T(\ba; \C P)  - q^{-ms} S_q(\B a) v_{\C P}(\bb) = \sum_{\ol z \mmod{q}} e(\F F(\ol z; \B a/q)) H(q, \bb,\ol z),
    \end{align*}
    where
    \begin{align*}
         H(q, \bb,\ol z) &= \dsum{\ol w \in \mathbb Z^{ms}}{ q \ol w + \ol z  \in \C P^s}   e(\F F(q \ol w + \ol z; \bb)) -   q^{-ms} \int_{\mathcal P^s} e ( \F F(\ol{\bm{\zeta}};  \bb)) \D  \ol{\bm{\zeta}}.
		\end{align*}
		Including the first term of $H(q, \bb,\ol z)$ in the integral, we may apply the Mean Value Theorem~and obtain
		\begin{align*}
         H(q, \bb,\ol z)&= \dsum{\ol w \in \mathbb Z^{ms}}{ q \ol w + \ol z  \in \C P^s} \int_{\ol w}^{\ol w+1} \left( e(\F F(q \ol w + \ol z; \bb))-e ( \F F(q \ol{\bm{\zeta}}+ \ol z;  \bb)) \right) \D  \ol{\bm{\zeta}}  + O(1)\\
         & \ll 1+\left(\frac{(Pq^{-1})^m}{\det C }\right)^{\hspace{-1mm} s} \left(\sum_{\B j \in J} |\beta_{\B j}|\frac{P^d}{\hgam}\right)\frac{q}{P_{\mathrm{min}}}.
    \end{align*}
    The remaining term is just $S_q(\B a)$ and can be bounded trivially by $q^{ms}$, so altogether we have
    \begin{align*}
         T(\ba; \C P)  - q^{-ms} S_q(\B a) v_{\C P}(\bb )   & \ll q^{ms} \left( 1+ \left( \frac{(Pq^{-1})^m}{\det C }\right)^{\hspace{-1mm} s} \left(\sum_{\B j \in J} |\beta_{\B j}|\frac{P^d}{\hgam}\right) \frac{q}{P_{\mathrm{min}}}\right)
    \end{align*}
    as claimed.
\end{proof}

We define the truncated singular series and singular integral as
\begin{align*}
    \F S_{\psi}(P) &= \sum_{q=1}^{c'P^{r(d-1)\theta}} q^{-ms} \dsum{\B a=0}{(\B a, q)=1}^{q-1} S_q(\B a) e \left(\frac{- \B a \cdot \B n}{q} \right)
\end{align*}
and
\begin{align*}
    \F J_{\psi}(\C P) &= \int_{|\beta_{\B j}| \le c' P^{-d+r(d-1)\theta}\hgam }v_{\C P}(\bb) e(- \bb \cdot \B n) \D \bb,
\end{align*}
respectively. This notation allows us to determine the overall error arising from substituting $q^{-ms}S_q(\B a) v(\bb)$ for $T(\ba; \C P)$ by integrating the expression from Lemma~\ref{lem3:gen} over $\F M'(\theta)$.

\begin{lem}\label{lem3:generror}
   The total major arcs contribution is given by
    \begin{equation*}
        \int_{\F M'(P,\theta)} T(\ba; \C P) e(- \ba \cdot \B n) \D  \ba = \F S_{\psi}(P) \F J_{\psi}(\C P) + O\hspace{-1mm}\left(\hspace{-1mm} \left(\frac{P^m}{\det C }\right)^{\hspace{-1mm} s-rd/m}P^{(d-1)r(2r+3)\theta -1+\eta}\hspace{-1mm}\right)\hspace{-1mm}.
    \end{equation*}
\end{lem}
\begin{proof}
    We have 
		\begin{align*}
            \vol \F M'(\theta)       &\ll  \sum_{q=1}^{c'P^{r(d-1)\theta}} \left(q P^{-d+r(d-1)\theta }\right)^{r} \prod_{{\B j} \in J}\hgam   \ll \left( \frac{P^m}{\det C }\right)^{\hspace{-1mm} -rd/m} P^{(d-1)r(2r+1)\theta} ,
    \end{align*}
    where the last step follows by \eqref{prod-hhgam}. On the other hand, inserting the major arcs inequalities~\eqref{major'} for $q$ and $\bb$ in the bound from Lemma~\ref{lem3:gen} yields
		\begin{align*}
         |T(\ba; \C P) - q^{-ms} S_q(\B a) v_{\C P}(\bb )| \ll \left(\frac{P^m}{\det C }\right)^{\hspace{-1mm} s} P^{-1+2r(d-1)\theta + \eta}.
		\end{align*}
		This implies the statement.
\end{proof}

We may now fix $\theta$ such that 
\begin{align*}
    0 < \theta \le \frac{1 - \eta - \delta}{(d-1)r(2r+3)}
\end{align*}
holds for some small $\delta>0$. Under this condition, Lemmata \ref{lem3:pruning} and \ref{lem3:generror} can be combined to establish
\begin{align*}
    N_\psi(\C P) = \F S_{\psi}(P) \F J_{\psi}(\C P) + O\left(\left(\frac{P^m}{\det C }\right)^{\hspace{-1mm} s-rd/m}P^{-\delta}\right)\hspace{-0.5mm}.
\end{align*}

As in \cite{FRF} the truncated singular series and integral can be extended to infinity. As usual, the expected growth rate is encoded in the geometry of the problem and derives from normalising the singular integral.
In fact, standard computations reveal that
\begin{align}\label{singint}
    \F J_{\psi}(\C P)     &= \left(\frac{P^{m}}{\det C }\right)^{\hspace{-1mm} s-rd/m} \int_{|\bb| \le c'P^{r(d-1)\theta} } v_{1}(\bb) e(- \bb \cdot \tilde{ \B n})  \D  \bb,
\end{align}
where $\tilde n_{\B j} = P^{-d}\hgam n_{\B j}$ for every $\B j \in J$. 

As in \cite{FRF} the truncated singular series and integral can be extended to infinity. In particular, $v_{1}(\bb)$ and $S_q(\B a)$ are independent of $C$ and are therefore identical to the respective quantities considered in our former work \cite{FRF}. 
We present here a slight refinement of the treatment given in \cite{FRF}, leading to our improved bounds on the number of variables that is required in order to obtain an asymptotic formula. For the sake of reference, we quote here Lemma 3.5 of \cite{FRF}.

\begin{lem}\label{ramFRF}
  Suppose $C=\Id_m$, and $k>0$ and $0 < \theta \le d/(2(d-1))$ are parameters with $s-\dim \sing F > 2^{d-1}k$.
 Then one of the following is true.
  \begin{enumerate}[(A)]
      \item\label{A} $\ba \in \F m(Q,\theta)$, i.e. $T(\ba, Q) \ll Q^{ms-k\theta}$, or
      \item\label{B} $\ba \in \F M(Q,\theta)$, i.e. for every $\B j \in J$ one finds integers $0 \le a_{\B j} < q_{\B j} \ll Q^{(d-1)\theta}$ and real numbers $|\beta_{\B j}| \ll q_{\B j}^{-1}Q^{-d+(d-1)\theta}$ satisfying $\alpha_{\B j} = a_{\B j}/q_{\B j} + \beta_{\B j}$, and this representation is unique.
  \end{enumerate}
\end{lem}
This allows us to understand the contribution from the singular series and the singular integral.

\begin{lem}\label{conv}
    The function $S_{q}(\B a)$ is pseudo-multiplicative, i.e. 
    \begin{align}\label{mult}
	S_{q}(\B a)S_{q'}(\B a') = S_{qq'}(q'\B a + q \B a')
    \end{align}
    whenever $(q,q')=1$. Furthermore, when $k>(d-1)(r+1+\delta)$ one has
    \begin{align}\label{ss-conv}
	q^{-ms}S_{q}(\B a) \ll q^{-(r+1+\delta)}
    \end{align}
    for every $\B a \mmod{q}$ with $(q,\B a) =1$, and 
    \begin{align}\label{si-conv}
	v_1(\bg) \ll |\bg|^{-(r+1+\delta)}.
    \end{align}
\end{lem}

\begin{proof}
    The multiplicativity follows by standard arguments, and the second statement is \cite[Lemma 6.3]{FRF} with $W=r+1+\delta$. Observe that in \cite{FRF}, only $d \ge 3$ is permitted, but a slight modification of the argument shows that the statement remains valid for $d=2$ as well.
    
    In the case of $v_1(\bg)$ the argument is similar. Let $\theta$ be given as above and choose $Q$ in such a way that $|\bg|=cQ^{(d-1)\theta}$. As before, the major arcs are disjoint, and the pont $P^{-d}\bg$ lies right on the edge with the approximation $\B a = \B 0$ and $q=1$. It therefore follows from Lemma \ref{lem3:generror} that
    \begin{align*}
	v_1(\bg)  \ll Q^{-ms} T(Q^{-d}\bg,Q) + Q^{-1} |\bg|.
    \end{align*}
    As before, by continuity the the minor arcs bound continues to apply on the boundaries of $\F m(Q,\theta)$ and yields a bound $T(Q^{-d}\bg, Q) \ll Q^{ms-k\theta}$. Altogether we thus obtain
    \begin{align*}
	v_1(\bg)  \ll Q^{-k\theta} + Q^{-1+(d-1)\theta} \ll |\bg|^{-(r+1+\delta)}.
    \end{align*}
    This completes the proof of the lemma.
\end{proof}

Now define the complete singular series $\F S_{\psi}$ and the singular integral $\F J_{\psi}$ as
\begin{align*}
     \F S_{\psi} &=\sum_{q=1}^{\infty} \dsum{\B a=0}{(\B a, q)=1}^{q-1} q^{-ms}S_q(\B a) e\left(\frac{- \B a \cdot \B n}{q}\right)
\end{align*}
and
\begin{align*}
     \F J_{\psi} &= \int_{\mathbb R^{r}} v_{\C P}(\bb) e(- \bb \cdot \B n)  \D  \bb.
\end{align*}
Then we have shown that 
\begin{align*}
		N_{\psi}(\C P) = \F J_{\psi}\F S_{\psi} + O\left(\left(\frac{P^{m}}{\det C }\right)^{\hspace{-1mm} s-rd/m} P^{-\delta}\right)\hspace{-.5mm},
\end{align*}
where the bound on the number of variables is obtained from combining \eqref{s>k} and \eqref{k-pruning} as well as the bound on $k$ given in Lemma \ref{conv} and is given by
\begin{align*}
    s - \dim \sing F > 2^{d-1} \max \left\lbrace 2(d-1), rd\frac{\sum_{i=1}^m \log P_i}{m \log P_{\mathrm{min}}}\right\rbrace \hspace{-.5mm},
\end{align*}
which is the bound claimed in Theorem~\ref{rep}.

It follows from \eqref{mult} that the singular series can be expanded as an Euler product 
$$\F S_{\psi}=\prod_p \chi_p(F; \psi)$$ 
whose factors are given by 
\begin{align*}
    \chi_p(F; \psi) &= \sum_{l=0}^{\infty}p^{-lms} \sum_{\substack{\B a =0 \\ (\B a, p)=1}}^{p^l-1} S_{p^l}(\B a) e\left(\frac{- \B a \cdot \B n}{p^l}\right)\hspace{-.5mm}.
\end{align*}
In particular, \eqref{ss-conv} implies that $\chi_p(F; \psi) = 1 + O(p^{-1-\delta})$, whence we may conclude that the Euler product converges and thus vanishes if and only if one of the factors is zero. 
Standard arguments now show that 
\begin{align*}
		\chi_p(F; \psi)  &= \lim_{l \to \infty} (p^l)^{r-ms} \card \{\ol x \mmod{p^l}: \Phi_{\B j}(\ol x) \equiv n_{\B j} \pmod{p^l} \quad  \forall \; \B j \in J\} \\
		& =\lim_{l \to \infty} (p^l)^{r-ms} \card \{\ol x \mmod{p^l}:  \\
		&\qquad \qquad \qquad \qquad \qquad F(\B x_1 t_1 + \dots + \B x_m t_m) \equiv \psi(t_1, \dots, t_m) \pmod{p^l}\}
\end{align*}
identically in $t_1, \dots, t_m$, where we used the correspondence between identical representations of forms and points on hypersurfaces as in \S \ref{p:setup}. Thus $\chi_p(F; \psi)$ measures the density of solutions to \eqref{rep} in $\mathbb Q_p$, and it follows from \cite[Theorem 5.1]{Hensel} in combination with \cite[Theorem 4.1]{tdw-local} that this density is positive for all $p$ whenever 
\begin{align*}
		ms-\dim \sing_m F \ge (d^2+m)^{2^{d-2}} d^{2^{d-1}},
\end{align*}
where 
\begin{align*}
		\sing_m F = \left\{ \ol x : \rank \left( \frac{\partial \Phi_{\B j} (\ol x)}{\partial \B x_i}  \right)_{\B j \in J, 1 \le i \le m} \le r-1 \right\} \hspace{-.5mm}, 
\end{align*}
and one probably has $\dim \sing_m F = (m-1)s + \dim \sing F$.\\

Similarly, standard computations show that the singular integral is given by
\begin{align*}
    \F J_{\psi} =\left(\frac{P^{m}}{\det C }\right)^{\hspace{-1mm} s-rd/m} \chi_{\infty}(\psi; P_1,\dots,P_m),
\end{align*}
where the integral
\begin{align}\label{chi_inf}
		\chi_{\infty}(\psi; P_1,\dots,P_m) = \int_{\mathbb R^{r}} \int_{|\ol{\bm{\xi}}| \le 1} e \left( \F F(\ol{\bm{\xi}};  \bb) - \bb \cdot \tilde{\B n}\right) \D  \ol{\bm{\xi}}  \D  \bb
\end{align}
measures solutions of a rescaled version of the problem in the real unit box. In fact, a straightforward generalisation of \S 11 and Lemma 2 of \cite{schmidtquad} shows that 
$$\chi_{\infty}(\psi; P_1,\dots,P_m) \gg 1$$ whenever 
$s-\dim \sing F> 2^{d}r(d-1)$
and the variety
\begin{align}\label{real-var}
		M = \{\ol x \in  (-1,1)^{ms}: \Phi_{\B j}(\ol x) = \tilde n_{\B j} \quad (\B j \in J) \} \subset \mathbb R^{ms}
\end{align}
is of dimension at least $ms-r$. As above, this variety may be rewritten in terms of $F$ and $\psi$ themselves. In fact, one has
\begin{align*}
		\sum_{\B j \in J} \tilde n_{\B j} t_{j_1} \cdot \ldots \cdot t_{j_d} &= \sum_{\B j \in J} P^{-d} n_{\B j} \gamma_{j_1} t_{j_1} \cdot \ldots \cdot \gamma_{j_d}t_{j_d}\\
		&= \sum_{\B j \in J} (t_{j_1}/P_{j_1}) \cdot \ldots \cdot (t_{j_d}/P_{j_d}) = \psi(\textstyle{\frac{t_1}{P_1}, \dots, \frac{t_m}{P_m}}),
\end{align*} 
whence $M$ is given by
\begin{align*}
		M = \{\ol x \in  (-1,1)^{ms}: F(\B x_1 t_1 + \dots + \B x_m t_m)= \psi(\textstyle{\frac{t_1}{P_1}, \dots, \frac{t_m}{P_m}}) \quad \forall \;t_1, \dots, t_m \},
\end{align*}
and the requirement is that this variety have a positive $(ms-r)$-dimensional volume. This condition will be satisfied as soon as $M$ contains at least one non-singular point. 

Whether or not $M$ contains any points depends on the choice of the parameters $P_i$ as well as the shape of $\psi$ itself. In the case of Theorem \ref{thm:gen} we pick $P_i = n_i^{1/d}$, so 
$$\psi(\textstyle{\frac{t_1}{P_1}, \dots, \frac{t_m}{P_m}}) = \widetilde \psi(t_1, \dots, t_m)$$
where $\widetilde \psi$ is as in \eqref{psi-til}
With this choice, we have $\tilde n_i = 1$ for $1 \le i \le m$, and $\tilde n_{\B j} \le 1$ for all $\B j \in J$ if $\psi$ is pseudo-diagonal. Observe that $F$ maps the unit cube to an area of size $O(1)$, so $M$ will not contain any points unless all coefficients of $\tilde \psi$ are at most of size $1$. In other words, if $\psi$ is not pseudo-diagonal up to a factor at most as large as the largest coefficient of $F$, the variety $M$ will invariably be empty.

\end{document}